\documentclass[a4paper,12pt]{article}
\usepackage[utf8]{inputenc}
\usepackage[english]{babel}
\usepackage{amssymb}
\usepackage{amsmath,amsthm,amssymb}
\usepackage{mathrsfs}
\usepackage{color}

\theoremstyle{plain}
\newtheorem{theorem}{Theorem}[section]
\newtheorem*{theorem*}{Theorem}
\newtheorem{lemma}[theorem]{Lemma}

\newtheorem{proposition}[theorem]{Proposition}
\newtheorem*{proposition*}{Proposition}

\theoremstyle{definition}
\newtheorem{definition}[theorem]{Definition}

\theoremstyle{remark}

\newcommand{\Ca}{{\db^{-1}}}
\newcommand{\Cab}{{\d^{-1}}}
\renewcommand{\d}{\partial}
\newcommand{\db}{\overline{\partial}}

\newcommand{\abs}[1]{\left\lvert #1 \right\rvert}
\newcommand{\norm}[1]{\left\lVert #1 \right\rVert}

\newcommand{\R}{\mathbb{R}}
\newcommand{\C}{\mathbb{C}}

\begin{document}
\title{Stability and uniqueness for
a two-dimensional inverse boundary value problem for
less regular potentials}

\author{E. Bl{\aa}sten\thanks{{Department of Mathematics, Tallinn University
of Technology, Ehitajate tee 5, 19086 Tallinn, Estonia,
e-mail: {\tt eemeli.blasten@iki.fi}}},
O.~Yu.~Imanuvilov\thanks{Department of Mathematics, Colorado State
University, 101 Weber Building, Fort Collins, CO 80523-1874, U.S.A.,
e-mail: {\tt oleg@math.colostate.edu}\
Partially supported by NSF grant DMS 1312900},
 and
M.~Yamamoto\thanks{ Department of Mathematical Sciences, The University
of Tokyo, Komaba, Meguro, Tokyo 153, Japan,
e-mail:
{\tt myama@ms.u-tokyo.ac.jp}}
}

\maketitle

\begin{abstract}
We consider inverse boundary value problems for the Schr\"odinger equations
in two dimensions.   Within less regular classes of potentials,
we establish a conditional stability estimate of logarithmic order.
Moreover we prove the uniqueness within $L^p$-class of potentials
with $p>2$.
\end{abstract}

In this paper, we prove stability estimates and the uniqueness for
an inverse boundary value problem for the  two-dimensional Schr\"odinger
equation within a class of less regular unknown potentials.
We refer to the first result Sylvester and Uhlmann \cite{SU}
in the case where dimensions are higher than or equal to three,
and since then many remarkable works concerning the
uniqueness have been published.  Here we do not intend to create 
a complete list of
publications and see e.g., a survey by Uhlmann \cite{Ul}.
In particular, the arguments in two dimensions are different from
higher dimensions and we refer to the uniqueness result by Nachman \cite{Na2},
and a stability estimate by Alessandrini \cite{Al}.
Also see Liu \cite{Liu}, and as survey on the uniqueness
mainly in two dimensions, see Imanuvilov and Yamamoto
\cite{IY2}.
So far all these estimates have had a logarithmic modulus of continuity, which
is no surprise because Mandache showed that this is the best one could expect
\cite{Man}.
The other fact is that most of the above mentioned work was done
for the conductivity equation, and so there were not many papers
on inverse boundary value problems for the Schr\"odinger equation with a
potential in two dimensions.
The result on uniqueness in this paper (Theorem \ref{L2}) was
announced by a pioneering contribution (Bukhgeim \cite{Bu}) that has 
led to many developments in the study of two dimensional inverse boundary 
value problems. However, his proof only gives uniqueness for potentials in the
class $W^1_p$ as pointed out in Bl{\aa}sten's licentiate thesis \cite{BlaLic}.
See also Novikov and Santacesaria \cite{NS1}, which proved stability assuming
some smoothness and \cite{NS2} which showed also a reconstruction formula.
Santacesaria \cite{San} continued working on stability, and
showed that the smoother it is, the better exponent there will be on the 
logarithm.

There are not many results about stability and uniqueness for
less regular potentials and we refer to  Bl{\aa}sten \cite{Bla},
and Imanuvilov and Yamamoto \cite{IY}.
The former is the doctoral thesis  of the first named author
and proved conditional stability under some a priori boundedness 
of unknown potentials, and the latter proved 
the uniqueness in determining $L^p$-potentials with $p>2$.

In this paper we prove the uniqueness result announced by Bukhgeim for $L^p$
potentials, $p>2$, and in addition give logarithmic type stability estimates
for potentials in the class $W^s_2$, $s \in (0,1]\setminus\{\frac{1}{2}\}$.
After \cite{Bla} and \cite{IY}, the authors recognized 
that an improvement and simplification of the proofs are possible.
That is, the main purpose of this paper is to improve the stability 
estimates obtained in \cite {Bla} and simplify the proof of \cite{IY} 
by using a unified method.

The paper is composed of six sections.
In Section 2, we formulate our inverse problem and in Section 3 we state
two main results Theorems 2.1 on the conditional stability and
Theorem 2.2 on the uniqueness and compare them with the results in 
\cite{Bla} and \cite{IY}.  Sections 3-6 are devoted for completing the proofs
of Theorems 2.1 and 2.2.

\section{Formulation}
\label{introSect}
Let $X \subset \R^2$ be a bounded domain with boundary
$\partial X$ of $C^{\infty}$-class.
Although it is possible to relax the regularity of the boundary for example
to a Lipschitz domain, we assume $C^{\infty}$-boundary for simplicity.
Moreover let $q \in L^p(X)$, $p>2$, be a potential function.
Consider the Schr\"odinger operator with the potential $q$ in the domain
$X$
\[
L_q(x,D)u := \Delta u + q u.
\]

We define define the \emph{Cauchy data} $\mathcal{C}_q$ by
\begin{definition}
{\it Let $X \subset \R^2$ be a bounded domain with smooth boundary
$\partial X$ and $q \in
L^p(X)$ with $p > 1$. Then
\[
\mathcal{C}_q = \{ (u, \partial_\nu u) \in W_2^{1/2}(\partial X) \times
W_2^{-1/2}(\partial X);\thinspace
L_q(x,D) u = 0, u \in W^{1}_2(X)\}.
\]}
\end{definition}
If zero is not an eigenvalue of the operator $L_q(x,D)$ with the
zero Dirichlet boundary conditions, then the Cauchy data are
equivalent to the Dirichlet-to-Neumann map $\Lambda_q$ defined by
$$
\Lambda_q f = \frac{\partial u}{\partial\nu}\vert_{\partial X},
\quad f \in W^{1/2}_2(\partial X),
$$
where $u \in W^1_2(X)$ is a unique solution to
$L_q(x,D)u = 0$ in $X$ and $u\vert_{\partial X}
= f$.

The paper is concerned with a variant of the  classical Calder\'on problem:
{\it Suppose that for two potentials $q_1$ and $q_2$ the corresponding Cauchy
data are equal. Does that imply the uniqueness of the potentials?}

The inverse problem asks whether the mapping $q \mapsto \mathcal{C}_q$ is
invertible.  The uniqueness means that no two different potentials
$q$ have the same Cauchy data $\mathcal{C}_q$.   The stability means
that the mapping inverse to $q \mapsto \mathcal{C}_q$ is continuous in some 
topologies.  For formulating the
stability, we define the difference of Cauchy data by
$$
d(\mathcal{C}_{q_1}, \mathcal{C}_{q_2})
:= \sup_{(u_1,u_2)\in \mathcal X_{q_1}\times \mathcal X_{q_2}}
\abs{\int_X u_1(q_1-q_2)u_2dx} ,
$$
where
$$
\mathcal X_{q}=\{ u \in W^1_2(X);\thinspace L_q(x,D)u = 0, \thinspace
\norm{u}_{W^1_2(X)} = 1 \}.
$$
The difference $d(\mathcal{C}_{q_1}, \mathcal{C}_{q_2})$ is not a metric,
but if $\mathcal{C}_{q_1} = \mathcal{C}_{q_2}$ then $d(\mathcal{C}_{q_1}, 
\mathcal{C}_{q_2}) = 0$.  Moreover if zero is not an eigenvalue of
the operator $L_{q_j}(x,D)$, $j=1,2$ with the zero Dirichlet boundary
condition,  then
\[
d(\mathcal{C}_{q_1}, \mathcal{C}_{q_2}) \leq C \norm{\Lambda_{q_1} 
- \Lambda_{q_2}}_{\mathcal L(W_2^{1/2}(\partial X);W_2^{-1/2}(\partial X))}
\]
by Lemma 3.2 proved below.  Here the right-hand side denotes
the operator norm.
This inequality means that for given
$\mathcal{C}_{q_1}$ and $\mathcal{C}_{q_2}$, without knowing
$q_1, q_2$ in $X$, it is possible to calculate
an upper bound for $d(\mathcal{C}_{q_1}, \mathcal{C}_{q_2})$.

Usually one can show only \emph{conditional stability}, which means stability
under some assumptions on norms of unknown potentials  $q$'s.
Other important topic is the reconstruction of a potential. That is,
given a Cauchy data, reconstruct the potential using an explicit algorithm,
and an even more valuable goal is to reconstruct $q$ in a stable way by given
noisy data about $\mathcal{C}_q$.
As for the reconstruction of less regular potentials, see
Astala, Faraco and Rogers \cite{Ast}, which shows a reconstruction formula for
potentials in $W_2^{1/2}$, and proves that
there exists a set of positive measure
where the reconstruction does not converge pointwise for less regular
potentials.  Our proof suggests that the reconstruction converges in the
$L^2$-norm and we here do not discuss details.

{\bf Notations.} Let $i=\sqrt{-1}$, $x=(x_1,x_2), x_1, x_2 \in \mathbb{
R}^1$, $z=x_1+ix_2$ and $\overline{z}$ denote the complex conjugate of
$z \in \mathbb{C}$. We identify $x  \in \mathbb{R}^2$ with $z = x_1 +ix_2
\in \mathbb{C}$ and $\xi=(\xi_1,\xi_2)$ with $\zeta=\xi_1+i\xi_2$.
We set $\partial_z = \frac 12(\partial_{x_1}-i\partial_{x_2})$,
$\partial_{\bar z} =
\frac12(\partial_{x_1}+i\partial_{x_2}).$   By $\mathcal L(Y_1,Y_2)$
we denote the space of linear continuous operators from a Banach
space $Y_1$ into a Banach space $Y_2$.
Let $B(0,\delta)$ be a ball in $\mathbb{R}^2$ of radius $\delta$
centered at $0.$
We define the Fourier transform by
$(\mathscr{F}u)(\xi) = \int_{\R^2} u(x)e^{-i(x,\xi)} dx$. \\
\vspace{0.3cm}
\bigskip

\section{Main results}
Henceforth $C>0$ denotes generic constants which are dependent on
$X$ and constants $s, M$, but independent of parameters $\tau$,
where $s, M, \tau$ are given later.

We here state our two main results.

\begin{theorem}\label{L1}
Let $X \subset \R^2$ be a bounded domain with smooth boundary
$\partial X$ and $s \in (0,1]\setminus\{\frac 12\}$.  We assume that $q_1, q_2
\in W^{s}_2(X)$ satisfy an a priori estimate
$\Vert q_j\Vert_{W^{s}_2(X)} \leq M$ with $M < \infty$
and $q_1-q_2 \in \mathaccent'27{W}^s_2(X)$.
Then there exists a constant $C>0$ such that
\[
\Vert q_1 - q_2\Vert_{L^2(X)} \le
\left\{
\begin{array}{lll}
C\left(1+ \ln \frac{1}{d(\mathcal{C}_{q_1}, \mathcal{C}_{q_2})}\right)
^{-s/2}, \quad
&\mbox{if $d(\mathcal{C}_{q_1}, \mathcal{C}_{q_2}) < 1$},\\
Cd(\mathcal{C}_{q_1}, \mathcal{C}_{q_2}), \quad &\mbox{if
$d(\mathcal{C}_{q_1}, \mathcal{C}_{q_2})\ge 1$}.
\end{array}
\right.
\]
\end{theorem}
Note that when $s<\frac{1}{2}$ no boundary behaviour
is required from the two potentials
(e.g., Adams and Fournier \cite{Adm}, Lions and Magenes \cite{LM}).

In our stability result, we estimate the norm $\Vert q_1-q_2\Vert_{L^2(X)}$ 
under the a priori boundedness of the norm in $\mathaccent'27{W}^s_2(X)$,
while the work \cite{Bla} uses different norms for $q_1-q_2$ and
a priori boundedness and for the norm.  As for the exponent in the 
estimate, our result asserts $-s/2$ which is better than 
$-s/4$ in \cite{Bla}, but it is still controlled by 
a logarithmic rate.

By the theorem \ref{L1}, we see that
$$
\Vert q_1 - q_2\Vert_{L^2(X)}
= O\left(
\left( \ln \frac{1}{d(\mathcal{C}_{q_1}, \mathcal{C}_{q_2})} \right)
^{-s/2} \right)
$$
as $d(\mathcal{C}_{q_1}, \mathcal{C}_{q_2}) \longrightarrow 0$.
Thus the rate of the conditional stability is logarithmic.

By Lemma 3.2 below, from Theorem 2.1, we can derive
\\
{\bf Corollary.}
{\it Under the same assumptions of Theorem 2.1, we further assume that
zero is not an eigenvalue of $L_{q_j}(x,D)$ with the zero Dirichlet boundary
condition.  Let $s \in (0,1]$, and let $q_1, q_2
\in W^{s}_2(X)$ satisfy
$\Vert q_j\Vert_{W^{s}_2(X)} \leq M$ with $M < \infty$
{and $q_1-q_2 \in \mathaccent'27{W}^s_2(X)$}.
Then there exists a constant $C>0$ such that
$$
\Vert q_1 - q_2\Vert_{L^2(X)}
\le  \left\{
\begin{array}{lll}
C\left( 1 + \ln \frac{1}{\Vert \Lambda_{q_1} - \Lambda_{q_2}\Vert
}
\right)^{-s/2}, \thinspace&\mbox{if}\thinspace
\Vert \Lambda_{q_1} - \Lambda_{q_2}\Vert
< 1,\\
C\Vert \Lambda_{q_1} - \Lambda_{q_2}\Vert,
\thinspace &\mbox{if}\thinspace \Vert \Lambda_{q_1} - \Lambda_{q_2}\Vert
\ge 1.
\end{array}
\right.
$$
where $\Vert \Lambda_{q_1} - \Lambda_{q_2}\Vert$ is the norm in 
${\mathcal{L}(W_2^{1/2}(\partial X); W_2^{-1/2}(\partial X))}$.
}

Our second main result is the uniqueness in the recovery of 
the potential for the Schr\"odinger operator :
\begin{theorem}\label{L2}
Let $X \subset \R^2$ be a bounded smooth domain and $q_1, q_2
\in L^{p}(X)$ with $p > 2$. If $\mathcal{C}_{q_1} = \mathcal{C}_{q_2}$,
then $q_1 = q_2$.
\end{theorem}

The merits for the proof of our unified method are as follows. 
\begin{enumerate}
\item
The proofs of both stability and uniqueness are simplified. 
Bl{\aa}sten \cite{Bla} used
Sobolev spaces where the $L^p$-norm has been replaced by a Lorentz-norm.
We can avoid using the Lorentz-norm by showing a Carleman estimate formulated
using conventional $L^p$-spaces. 
\item
Comparing with Imanuvilov and Yamamoto \cite{IY},
we use a simpler $L^2$-convergent stationary-phase argument which avoids
approximating the potentials by test functions and using Egorov's theorem.
\end{enumerate}

\section{Key lemmas and definitions}
We start this section with the following Lemma:

\begin{lemma}
\label{sameCauchyDataLemma}
Let $X \subset \R^2$ be a bounded Lipschitz domain and $q_1, q_2 \in L^p(X)$, 
$p > 1$. If $\mathcal{C}_{q_1} = \mathcal{C}_{q_2}$, then
\[
\int_X u_1(q_1-q_2)u_2dx = 0
\]
for all $(u_1,u_2)\in \mathcal X_{q_1}\times \mathcal X_{q_2}.$
\end{lemma}

\medskip

\begin{lemma}
\label{boundaryDataDiffLemma}
Let $X \subset \R^2$ be a bounded smooth domain and $q_1, q_2 \in
L^p(X)$, $p>1$ be potentials.   We  assume that $0$ is not an
eigenvalue of the operator $L_{q_j}(x,D)$, $j=1,2$, with the zero Dirichlet
boundary condition.  Then
$$
d(\mathcal{C}_{q_1}, \mathcal{C}_{q_2})
\le
\norm{\operatorname{Tr}}_{\mathcal L(W_2^1(X); W_2^{1/2}(\partial X))}
^2 \norm{ \Lambda_{q_1} - \Lambda_{q_2}}_{\mathcal L(W_2^{1/2}(\partial X);
W_2^{-1/2}(\partial X))}.
$$
\end{lemma}
\begin{proof}
Let $u_1,U \in W_2^1(X)$ satisfy $L_{q_1}(x,D) U = L_{q_1}(x,D) u_1 =0$ in $X$ 
and
$U =  u_2$ on $\partial X$.
Then
$$
\Delta (U-u_2) + q_1(U-u_2) + (q_1-q_2)u_2 = 0 \quad \mbox{in $X$}
$$
and $U-u_2 = 0$ on $\partial X$.  Multiplying by $u_1$, integrating by
parts and using $\Delta u_1 + q_1u_1 = 0$ in $X$ and
$U-u_2 = 0$ on $\partial X$, we have
$$
\int_X u_1(q_1-q_2)u_2dx
= \int_{\partial X} \partial_\nu (u_2 - U)u_1d\sigma.
$$
Now note that $( U, \partial_\nu U) \in \mathcal{C}_{q_1} $ and 
$( u_2, \partial_\nu u_2) \in \mathcal{C}_{q_2}$. This observation allows us 
to switch to the Dirichlet-to-Neumann maps, and so
\begin{multline*}
\abs{\int_{\partial X}(\partial_\nu u_2 - \partial_\nu U) u_1 d\sigma } \\
= \abs{\int_{\partial X}(\Lambda_{q_2} u_2 - \Lambda_{q_1} U)  u_1d\sigma}
= \abs{\int_{\partial X} ((\Lambda_{q_2} - \Lambda_{q_1}) u_2)u_1 d\sigma}
\end{multline*}
because $u_2 =  U$ on $\partial X.$ Now take the supremum over $(u_1,u_2)\in 
\mathcal X_{q_1}\times \mathcal X_{q_2},$ to obtain
\begin{multline*}
\sup_{(u_1,u_2)\in \mathcal X_{q_1}\times \mathcal X_{q_2}}  
\abs{\int_X u_1(q_1-q_2)u_2dx} \\
= \sup_{(u_1,u_2)\in \mathcal X_{q_1}\times \mathcal X_{q_2}} 
\abs{\int_{\partial X}((\Lambda_{q_2} - \Lambda_{q_1}) u_2) u_1d\sigma} \\
\leq \norm{\operatorname{Tr}}_{\mathcal L(W_2^1(X);W_2^{1/2}(\partial X))}^2 
\norm{\Lambda_{q_1} - \Lambda_{q_2}}_{\mathcal L(W_2^{1/2}(\partial X); 
W_2^{-1/2}(\partial X))}.
\end{multline*}
The proof of Lemma \ref{boundaryDataDiffLemma} is complete.
\end{proof}

Henceforth we identify $z_0 = x_{01}+ix_{02} \in \C$ with
$x_0 = (x_{01}, x_{02}) \in \R^2$.

The following lemma plays the important role in the proof of Theorems 2.1
and 2.2.
\begin{lemma}
\label{statPhaseLemma}
Let $\tau > 0$, $0 \leq s \leq 1$ and $Q \in W^{s}_2(\R^2)$, $z_0
\in \C$. Then
\begin{equation}\label{ananas}
\norm{Q - \int_{\R^2} \frac{2 \tau}{\pi} e^{\pm i \tau \left((z-z_0)^2
+ (\overline{z-z_0})^2\right)}Q dx}_{L^2(\R^2;dx_0)}
\leq 2 \tau^{-s/2} \norm{Q}_{W^{s}_2(\R^2)}.
\end{equation}
If $s = 0$, then the left-hand side tends to $0$ as $\tau \to \infty$.
\end{lemma}
\begin{proof}

First for $\delta>0$, we have
\begin{eqnarray*}
&&\theta_{\delta}(\xi)
:= \mathscr{F}(e^{\pm i\tau(z^2 + \overline{z}^2) - \delta\vert z\vert^2})
(\xi) \\
= && \frac{\pi}{\sqrt{\delta^2+4\tau^2}}
\exp\left(-\frac{\delta\vert\xi\vert^2}{16\tau^2+4\delta^2}\right)
\exp\left( \frac{\mp i\tau(\xi_1^2-\xi_2^2)\tau}{8\tau^2+2\delta^2}\right).
\end{eqnarray*}
The calculations are direct and we refer to
pp.210-211 in Evans \cite{Ev} for example.  Let $\mathcal{S}(\R^2)$ be the
space  rapidly decreasing functions and $\mathcal{S}'(\R^2)$ be the
dual, that is, the space of tempered distributions.
Since
$$
\theta_{\delta} \longrightarrow
\frac{\pi}{2\tau}
\exp\left( \frac{\mp i(\xi_1^2-\xi_2^2)}{8\tau}\right)
= \frac{\pi}{2\tau}\exp\left( \mp \frac{i(\zeta^2 + \overline{\zeta}^2)}
{16\tau}
\right)
$$
and
$$
e^{\pm i\tau(z^2 + \overline{z}^2) - \delta\vert z\vert^2}
\longrightarrow e^{\pm i\tau(z^2 + \overline{z}^2)}
$$
as $\delta \downarrow 0$ in $\mathcal{S}'(\R^2)$ and
$\mathscr{F}$ is continuous from $\mathcal{S}'(\R^2)$ to itself, we
see
$$
\mathscr{F}(e^{\pm i\tau(z^2 + \overline{z}^2)})(\xi)
= \frac{\pi}{2\tau}\exp\left( \mp \frac{i(\zeta^2 + \overline{\zeta}^2)}
{16\tau}
\right)
$$
in $\mathcal{S}'(\R^2)$.  This equality holds for almost all $\xi \in \R^2$,
because the right-hand side is in $L^{\infty}(\R^2)$.

Next let $Q \in C^{\infty}_0(\R^2)$ be arbitrarily chosen.  Then
$$
\mathscr{F}\left(\frac{2\tau}{\pi}e^{\pm i\tau(z^2 + \overline{z}^2)}
* Q\right)
= \exp\left( \frac{\mp i(\zeta^2 + \overline{\zeta}^2)}{16\tau}\right)
\mathscr{F}(Q)(\xi).
$$
Hence by the Plancherel theorem, we have
\begin{eqnarray*}
& \left\Vert Q - \frac{2\tau}{\pi} e^{\pm i\tau(z^2 + \overline{z}^2)} * Q
\right\Vert_{L^2(\R^2)}
= \frac{1}{2\pi} \left\Vert \mathscr{F}Q
- \mathscr{F}\left(
\frac{2\tau}{\pi} e^{\pm i\tau(z^2 + \overline{z}^2)} * Q\right)
\right\Vert_{L^2(\R^2)} \\
&
= \frac{1}{2\pi}\left\Vert
\left(1 - e^{\mp i\frac{\xi^2 + \overline{\xi}^2}{16 \tau}}\right)
\mathscr{F}Q
\right\Vert_{L^2(\R^2)}.
\end{eqnarray*}
On the other hand, we can prove
$$
\vert 1 - e^{\mp i(\zeta^2 + \overline{\zeta}^2)}\vert
\leq 2^{1+s/2} \abs{\xi}^s
$$
for $0 \le s \le 1$ and $\zeta \in \C$.  In fact, if $\vert \xi\vert \ge 1$,
then $\vert 1 - e^{\mp i(\zeta^2 + \overline{\zeta}^2)}\vert
\leq 2 \le 2^{1+s/2}$
and so the inequality is seen.  Let $\vert \xi\vert \le 1$.
Direct calculations yield
$\vert 1 - e^{\mp i(\zeta^2 + \overline{\zeta}^2)}\vert^2
= 4\sin^2(\xi_1^2 - \xi_2^2)$.  Therefore
$$
\vert 1 - e^{\mp i(\zeta^2 + \overline{\zeta}^2)}\vert^2
\le 4\vert \xi_1^2 - \xi_2^2\vert^2
\le  4\vert \xi_1^2+\xi_2^2\vert^2
\le 4\times 2^s\vert \xi\vert^{2s},
$$
where we used $0 \le s \le 1$ and $\vert \xi\vert \le 1$.
Thus we have seen
$\vert 1 - e^{\mp i(\xi^2 + \overline{\xi}^2)}\vert
\leq 2^{1+s/2} \abs{\xi}^s$ for $0 \le s \le 1$ and $\xi \in \C$.

Hence
\begin{eqnarray}\label{morkovka}
&&\left\Vert Q - \frac{2\tau}{\pi} e^{\pm i\tau(z^2 + \overline{z}^2)} * Q
\right\Vert_{L^2(\mathbb{R}^2)}
\le \frac{1}{\pi}2^{s/2}
\left\Vert \left( \frac{\xi}{4\sqrt{\vert\tau\vert}}\right)^s
\mathscr{F}Q\right\Vert_{L^2(\R^2)}\nonumber\\
\le && \frac{1}{\pi}2^{s/2}2^{-2s}\vert \tau\vert^{-s/2}
\Vert (1+\vert \xi\vert^2)^{s/2} \mathscr{F}Q\Vert_{L^2(\R^2)}.
\end{eqnarray}
for each $Q \in C^{\infty}_0(\R^2)$.  Since $C^{\infty}_0(\R^2)$ is dense in
$W^s_2(\R^2)$, passing to the limits, we complete the
proof of Lemma 3.3 for $s>0.$ If $s=0$ and $Q\in L^2(\mathbb{R}^2)$ for any 
positive $\epsilon$  we take  a function  $Q_\epsilon
\in C^\infty_0(\mathbb{R}^2)$ such that $\Vert Q-Q_\epsilon\Vert
_{L^2(\mathbb{R}^2)}\le \epsilon.$ Then (\ref{morkovka}) implies that 
for any positive $\tau$
$$
\left\Vert Q-Q_\epsilon - \frac{2\tau}{\pi} e^{\pm i\tau(z^2 
+ \overline{z}^2)} * (Q-Q_\epsilon)\right\Vert_{L^2(\mathbb{R}^2)}
\le \frac{1}{\pi}\left \Vert Q-Q_\epsilon\right\Vert_{L^2(\mathbb{R}^2)} 
\le \epsilon.
$$
Then applying to the function $Q_\epsilon$ estimate (\ref{ananas}), 
we obtain the statement of our lemma for $s=0.$
\end{proof}

\section{Preliminary estimates}
Let us introduce the
operators:
$$
\bar \partial^{-1}g=-\frac 1\pi\int_X
\frac{g(\xi_1,\xi_2)}{\zeta-z}d\xi_1d\xi_2,\quad
\partial^{-1}g=-\frac 1\pi\int_X
\frac{g(\xi_1,\xi_2)}{\overline\zeta-\overline z}d\xi_1d\xi_2,
$$
where $X\subset \mathbb{R}^2$ is a bounded domain with the smooth boundary.

We have
\begin{proposition}\label{Proposition 3.0}
 {\bf A}) Let $1\le p\le 2$ and $ 1<\gamma<\frac{2p}{2-p}.$ Then
 $\bar\partial^{-1},\partial^{-1}\in
\mathcal L(L^p( X),L^\gamma(X)).$
\newline
{\bf B})Let $1< p<\infty.$ Then  $\bar\partial^{-1},
\partial^{-1}\in \mathcal L(L^p( X),W^1_p(X)).$
\end{proposition}
A) is proved on p.47 in \cite{VE} and B) can be verified
by using Theorem 1.32 (p.56) in \cite{VE}.
{$\blacksquare$}

\medskip
Henceforth for arbitrarily fixed $z_0 \in \C$, we set
$$
\Phi(z) = \Phi(z;z_0) := (z-z_0)^2
$$
and introduce the operator:
$$
\widetilde {\mathcal R}_{\tau}g = \frac 12 e^{-i\tau(\overline \Phi+\Phi)}
\partial^{-1}(ge^{i\tau(\Phi+\overline \Phi)}).
$$

We set
\begin{eqnarray}\label{lodka}
&& U_0=1,  \quad
U_1=\widetilde{\mathcal R}_\tau(\frac 12(\bar\partial^{-1}q
-\bar\partial^{-1}q(x_0))),\\
&& U_j=\widetilde
{\mathcal R}_\tau(\frac 12\bar\partial^{-1}(qU_{j-1}))\quad \forall j\ge 2.
\end{eqnarray}
We construct a solution to the Schr\"odinger equation in the form
\begin{equation}\label{gavnuk}
 u_1=\sum_{j=0}^\infty e^{i\tau \Phi}(-1)^jU_j.
\end{equation}

Henceforth $C(\epsilon)$ denotes generic constants which are dependent on
not only $s,M,X$ but also $\epsilon$.

We will prove that the infinite series is convergent in
$L^r(X)$ with some $r > 2$.  For it, we show the following propositions.

\begin{proposition}\label{elka}
Let $u\in W^1_p(X)$ for any $p>2.$
Then for any $\epsilon\in(0,1)$ there exists a constant
$C(\epsilon)$  independent of $x_0\in X$ and $\tau$ such that
\begin{equation}\label{KJ}
\tau^{1-\epsilon}\Vert \widetilde{\mathcal R}_\tau u\Vert_{L^2(X)}
+ \tau^{1/p}\Vert \widetilde{\mathcal R}_\tau u\Vert_{L^\infty(X)}\le
C(\epsilon)\Vert u\Vert_{W^1_p(X)}\quad \forall \tau >0.
\end{equation}
\end{proposition}
{\bf Proof.} Let $\rho\in C_0^\infty(B(0,1))$ and
$\rho\vert_{B(0,\frac 12)}=1.$ We set
$\rho_\tau=\rho(\root\of{\tau} (x-x_0)).$ Since $\widetilde{\mathcal
R}_\tau u=\widetilde{\mathcal R}_\tau (\rho_\tau
u)+\widetilde{\mathcal R}_\tau ((1-\rho_\tau) u)$ for any positive
$\epsilon$, there exists $p_0(\epsilon)>1$ such that
 $\Vert e^{i\tau(\Phi+\bar\Phi)}\rho_\tau u\Vert_{L^{p_0(\epsilon)}(X)}\le
 C(\epsilon)\Vert u\Vert_{W^1_p(X)}/\tau^{1-\epsilon}.$ 
Moreover since $\Vert e^{i\tau(\Phi+\bar\Phi)}u\Vert_{L^{\infty}(X)}
\le C\Vert u\Vert_{W^1_p(X)}$ we have
$$
\Vert e^{i\tau(\Phi+\bar\Phi)}\rho_\tau u\Vert_{L^{\infty}(X)}\le
C(\epsilon)\Vert u\Vert_{W^1_p(X)}/\tau^{1-\epsilon}.
$$
Hence applying Proposition \ref{Proposition 3.0} and the Sobolev
embedding theorem, we have
\begin{equation}\label{J4}
\tau^{1-\epsilon}\Vert\widetilde{\mathcal R}_\tau (\rho_\tau u)\Vert_{L^2(X)}
+\tau^{1/p}\Vert\widetilde{\mathcal R}_\tau (\rho_\tau u)\Vert_{L^\infty(X)}
\le
C(\epsilon)\Vert u\Vert_{W^1_p(X)}, \quad
\forall\epsilon\in (0,1).
\end{equation}
Observe that
\begin{multline}\label{ippo}
\int_{
X} \frac{(1-\rho_\tau) u e^{i\tau(\Phi+\overline\Phi)}}{\overline
z-\overline \zeta}d\xi=\int_{X}\frac{(1-\rho_\tau) u\partial
e^{i\tau(\Phi+\overline\Phi)}}{\tau(\overline z-\overline
\zeta)i\partial\Phi}d\xi\\
= \int_{\partial X}\frac{(\nu_1-i\nu_2)(1-\rho_\tau) u
e^{i\tau(\Phi+\overline\Phi)}}{2i\tau(\overline z-\overline
\zeta)\partial\Phi}d\sigma \\
- \int_{X}\frac{1}{\tau(\overline
z-\overline \zeta)}\partial \left(\frac{(1-\rho_\tau)
u}{i\partial\Phi}\right
)e^{i\tau(\Phi+\overline\Phi)}d\xi +\frac{(1-\rho_\tau)u
e^{i\tau(\Phi+\overline\Phi)}}{i\tau
\partial\Phi}.
\end{multline}

Obviously, by the Sobolev embedding theorem, for any positive
$\epsilon$, there exists a constant $C(\epsilon)$ such that
\begin{equation}\label{J6}
\tau^{1-\epsilon}\left\Vert \frac{(1-\rho_\tau)u}{\tau
\partial\Phi}\right\Vert_{L^2(X)}+\tau^{1/2}\left\Vert 
\frac{(1-\rho_\tau)u}{\tau
\partial\Phi}\right\Vert_{L^\infty(X)}
\le C(\epsilon)\Vert u\Vert_{W^1_p(X)}.
\end{equation}
For the second term on the right-hand side of (\ref{ippo}), we have
$$
\left\vert\int_{X}\frac{1}{\tau(\overline z-\overline \zeta)}
\partial \left(\frac{(1-\rho_\tau) u}{\partial\Phi}
\right)e^{i\tau(\Phi+\overline\Phi)}d\xi\right\vert
\le\int_{X}\left\vert\frac{1}{\tau^\frac 12(\overline z-\overline
\zeta)}\left(\frac{\partial\rho(\root\of{\tau}\xi)
u}{\partial\Phi}\right )\right\vert d\xi
$$
$$
+\int_{X}\left\vert\frac{1}{\tau(\overline z-\overline \zeta)}
\left(\frac{(1-\rho_\tau)\partial u}{\partial\Phi}\right
)\right\vert d\xi + \int_{X}\left\vert\frac{2}{\tau(\overline
z-\overline \zeta)}\left(\frac{(1-\rho_\tau)
u}{(\partial\Phi)^2}\right )\right\vert d\xi .
$$
The functions $\frac{(1-\rho_\tau)\partial
u}{\partial\Phi}$ are uniformly bounded in $\tau$ in
$L^{p_1}(X)$ for some $p_1\in (1,2).$ 
Moreover, since $\Vert (1-\rho_\tau)/\partial\Phi\Vert_{L^\infty(X)}
\le C\root\of{\tau}$, the functions  
$\root\of{\tau}\frac{(1-\rho_\tau)\partial u}
{\partial\Phi}$ are uniformly bounded in $\tau$ in
functions  $\frac{(1-\rho_\tau)\partial
u}{\root\of{\tau}\partial\Phi}$ are uniformly bounded in $\tau$ in
$L^{p}(X).$ Applying Proposition
\ref{Proposition 3.0}, we have
\begin{equation}\label{J3}
\tau\left\Vert \partial^{-1} \left(\frac{(1-\rho_\tau)\partial_z
u}{\tau\partial\Phi}\right )\right\Vert_{L^2(X)} \!\!\!+\tau^{1/p}
\left\Vert \partial^{-1} \left(\frac{(1-\rho_\tau)\partial_z
u}{\tau\partial\Phi}\right )\right\Vert_{L^\infty(X)}\!\!\!\le C\Vert
u\Vert_{W^1_p(X)}.
\end{equation}
On the other hand, for any $p_2> 1$ we have
$$
\left\Vert\frac{\partial\rho(\root\of{\tau}\cdot)
u}{\partial\Phi}\right\Vert_{L^{p_2}(X)}\le C\Vert
u\Vert_{C^0(\overline X)} \left\Vert
\frac{1}{\partial\Phi}\right\Vert_{L^{p_2}(B(0,\frac{1}{\root\of{\tau}}))}
\le C \tau^{(2-p_2)/2p_2} \Vert u\Vert_{W^1_p(X)}.
$$

Thanks to this inequality, applying Proposition \ref{Proposition
3.0} again, we have:
\begin{eqnarray}\label{J2}
\tau^{1-\epsilon}\left \Vert\frac{1}{\tau^\frac 12}
\partial^{-1}\left(\frac{\partial\rho(\root\of{\tau}\cdot)
u}{\partial\Phi}\right )\right\Vert_{L^2(X)} &+\tau^{1/p}\left \Vert
\frac{1}{\tau^\frac 12}
\partial^{-1}\left(\frac{\partial\rho(\root\of{\tau}\cdot)
u}{\partial\Phi}\right )\right\Vert_{L^\infty(X)} \nonumber\\
& \le C(\epsilon) \Vert u\Vert_{W^1_p(X)}.
\end{eqnarray}
For any $p_3>1$, we have
\begin{eqnarray*}
&&\left\Vert
\frac{(1-\rho_\tau)u}{(\partial\Phi)^2}\right\Vert_{L^{p_3}(X)}\le
C\Vert u\Vert_{C^0(\overline X)}\left\Vert
\frac{1}{(\partial\Phi)^2}\right\Vert_{L^{p_3}(X\setminus
B(0,\frac{1}{2\root\of{\tau}}))}\\
&& \le C(p_3) \Vert u\Vert_{W^1_p(X)}\tau^{(2p_3-2)/2p_3}.
\end{eqnarray*}
Therefore
\begin{eqnarray}
\tau^{1-\epsilon}\left\Vert\partial^{-1}\left(\frac{(1-\rho_\tau)
u}{\tau(\partial\Phi)^2}\right )\right\Vert_{L^2(X)}&+\tau^{1/p}
\left\Vert\partial^{-1}\left(\frac{(1-\rho_\tau)
u}{\tau(\partial\Phi)^2}\right )\right\Vert_{L^\infty(X)}\nonumber\\
&\le
C(\epsilon)\Vert u\Vert_{W^1_p(X)}. \end{eqnarray}

From the classical representation of the Cauchy integral (see e.g. \cite{Mir} 
p.27)
we obtain
\begin{eqnarray}
\left\Vert \int_{\partial X}\frac{(\nu_1-i\nu_2)(1-\rho_\tau) u
e^{i\tau(\Phi+\overline\Phi)}}{2i\tau(\overline z-\overline
\zeta)\partial\Phi}d\sigma\right\Vert_{L^2(X)} \nonumber\\
\le C\left\Vert \frac{(\nu_1-i\nu_2)(1-\rho_\tau) u
e^{i\tau(\Phi+\overline\Phi)}}{2i\tau\partial\Phi}\right\Vert_{L^1(\partial X)}
                                         \nonumber \\
\le C\left\Vert\frac{(1-\rho_\tau)}{\partial\Phi}\right\Vert
_{L^1(\partial X)}\Vert
u\Vert_{W^1_p(X)}/\tau \le C\Vert
u\Vert_{W^1_p(X)}\ln \tau/\tau.
\end{eqnarray}

By the trace theorem and the Sobolev embedding theorem, for any $p>2$  there
exists a positive $\alpha=\alpha(p)$ such that the trace operator is  
continuous from $W^1_p(X)$ into $C^\alpha(\partial X).$
Using Theorem 1.11 (see p. 22 of \cite{VE}),
for any $\delta\in (0,\alpha(p))$, there exists a constant
$C(\delta)>0$ such that
\begin{eqnarray*}
\left\Vert \int_{\partial X}\frac{(\nu_1-i\nu_2)(1-\rho_\tau) u
e^{i\tau(\Phi+\overline\Phi)}}{2i\tau(\overline z-\overline
\zeta)\partial\Phi}d\sigma\right\Vert_{L^\infty(X)}
\\
\le C(\delta)\left\Vert \frac{(\nu_1-i\nu_2)(1-\rho_\tau) u
e^{i\tau(\Phi+\overline\Phi)}}{2i\tau\partial\Phi}
\right\Vert_{C^\delta(\partial X)}                   \nonumber \\
\le C(\delta)\Vert\frac{(1-\rho_\tau)}{\partial\Phi}
e^{i\tau(\Phi+\overline\Phi)}\Vert_{C^\delta(\partial X)}\Vert
u\Vert_{W^1_p(X)}/\tau.
\end{eqnarray*}
Denote $\mu_{\tau}(x)=\frac{(1-\rho_\tau)}{\partial\Phi}
e^{i\tau(\Phi+\overline\Phi)}.$ Then by the definitions of the functions
$\Phi$ and $\rho_\tau$ (noting that we identify $z_0$
with $x_0$), we estimate
$$
\Vert \mu_{\tau}(\cdot)\Vert_{C^0(\partial X)}
\le C\root\of{\tau}\quad \mbox{and}\quad
\Vert\nabla \mu_{\tau}(\cdot)\Vert_{C^0(\partial X)}
\le C\tau   \quad \forall \tau>1.
$$
Since in view of the mean value theorem, we can estimate
\begin{equation}
 \vert \mu_{\tau}(x) - \mu_{\tau}(x')\vert
= \vert \mu_{\tau}(x) - \mu_{\tau}(x')\vert^{1-\delta}
\vert \mu_{\tau}(x) - \mu_{\tau}(x')\vert^{\delta}
\le C\tau^{\frac{1-\delta}{2}}\tau^\delta\vert x-x'\vert^{\delta}
\end{equation}
and
we obtain
\begin{eqnarray}\label{J1}
\left\Vert \int_{\partial X}\frac{(\nu_1-i\nu_2)(1-\rho_\tau) u
e^{i\tau(\Phi+\overline\Phi)}}{2i\tau(\overline z-\overline
\zeta)\partial\Phi}d\sigma\right\Vert_{L^\infty(X)}
\le C(\delta)\Vert u\Vert_{W^1_p(X)}/\tau^{(1-\delta)/2}.
\end{eqnarray}
From (\ref{J4})-(\ref{J1}) we have (\ref{KJ}). $\blacksquare$

\bigskip
Now we proceed to the proof that the infinite series (\ref{gavnuk}) is
convergent in
$L^r(X)$ for all sufficiently large $\tau.$  Let $\tilde p\in
(2,p).$ By (\ref{KJ}) and Proposition \ref{Proposition 3.0}
and the H\"older inequality, there exists a positive constant
$\delta(\tilde p)$ such that
\begin{equation}\label{sun}
\Vert \widetilde{\mathcal R}_\tau u\Vert_{L^\frac{p\tilde
p}{p-\tilde p}(X)}\le C\Vert u\Vert_{W^1_{\tilde
p}(X)}/\tau^\delta.
\end{equation}
 Using (\ref{sun}) we have
\begin{multline}
\Vert U_{j}\Vert_{L^\frac{p\tilde p}{p-\tilde p}(X)}\le
\frac{C}{\tau^\delta}\Vert \frac 12\bar\partial^{-1}(qU_{j-1})\Vert
_{W^1_{\tilde p}(X)}\\
\le \frac{C}{2\tau^\delta} \Vert
\bar\partial^{-1}\Vert_{\mathcal L(L^{\tilde p}(X)
; W^1_{\tilde p}(X))}\Vert q U_{j-1}\Vert_{L^{\tilde
p}(X)}\\
\le\frac{C}{2\tau^\delta} \Vert
\bar\partial^{-1}\Vert_{\mathcal L(L^{\tilde p}(X)
;W^1_{\tilde p}(X))}\Vert q\Vert_{L^p(X)}\Vert
U_{j-1}\Vert_{L^\frac{\tilde pp}{p-\tilde p}(X)} \\
\le\left(\frac{C \Vert \bar\partial^{-1}\Vert
_{\mathcal L(L^{\tilde p}(X); W^1_{\tilde p}(X))}\Vert
q\Vert_{L^p(X)}}{2\tau^\delta}\right )^{j-1}\Vert
U_{1}\Vert_{L^\frac{p\tilde p}{p-\tilde p}(X)}.
\end{multline}
Therefore there exists $\tau_0$ such that for all $\tau>\tau_0$
$$
\Vert U_{j}\Vert_{L^\frac{p\tilde p}{p-\tilde p}(X)}\le
\frac{1}{2^j} \Vert U_{1}\Vert_{L^\frac{p\tilde p}{p-\tilde
p}(X)} \quad \forall j\ge 2.
$$
Hence the convergence of the series is proved.

\bigskip
Since
\begin{eqnarray*}
&&L_{q}(x,D)(U_je^{i\tau \Phi})=4\bar\partial\partial
(e^{i\tau \Phi}\widetilde R_\tau(\frac 12\bar\partial^{-1}
(qU_{j-1})))+qU_je^{i\tau \Phi}\\
=&& 2\bar\partial(e^{i\tau\Phi}\frac
12\bar\partial^{-1} (qU_{j-1}))+q_1U_je^{i\tau
\Phi}=qU_{j-1}e^{i\tau \Phi}+qU_je^{i\tau \Phi},
\end{eqnarray*}
the infinite series (\ref{gavnuk})
represents the solution to the Schr\"odinger equation. By
Proposition \ref{elka}, we have
\begin{equation}\label{kazi}
\left\Vert \sum_{j=2}^\infty (-1)^jU_j\right\Vert_{L^2(X)}
= O\left(\frac {1}{\tau^\frac 32}\right)\quad \mbox{as}\,\,\tau\rightarrow
+ \infty.
\end{equation}

Besides the estimate  (\ref{kazi}) we need the estimate of the infinite 
series  $ \sum_{j=2}^\infty (-1)^jU_j$ in the space $L^\infty(X).$

By Proposition \ref{elka}, we have
\begin{equation}\label{kazi1}
\left\Vert \sum_{j=2}^\infty (-1)^jU_j\right\Vert_{L^{\infty}(X)}
= O\left(\frac {1}{\tau^\frac 1p}\right)\quad \mbox{as}\,\,\tau\rightarrow 
+\infty.
\end{equation}

\begin{proposition}
\label{BukhSolExist2}
Let $q \in L^p(X)$ and $2<p<\infty.$ Then there exists a positive constant 
$\widehat C(\norm{q}_{L^p(X)})$  independent of $\tau$ and $x_0$ such that  
if $\tau >
\widehat C( \norm{q}_{L^p(X)})$ and $x_0 \in X$, then
there exists $u \in W^1_2(X)$
such that $L_q(x,D) u = 0$ in $X$ and
\begin{equation}\label{robot0}
u(x,x_0) = e^{i\tau\Phi} (1-\frac{1}{4}e^{-i\tau (\bar\Phi+\Phi)}
\partial^{-1}(e^{i\tau (\bar\Phi+\Phi)}(\bar\partial^{-1}q
-\bar\partial^{-1}q(x_0))) + r(x,x_0)),
\end{equation}
and there exists a positive constant $C_1$, independent of $\tau$ and
$x_0\in X$, such that
\begin{equation}\label{lom}
\tau^\frac 32 \sup_{x_0\in X}\Vert r(\cdot,x_0)\Vert_{L^2(X)}
+\tau^{\frac 12+\frac{1}{2p}} \sup_{x_0\in X}\Vert r(\cdot,x_0)\Vert_{L^4(X)}
\le C_1\Vert q\Vert_{L^p(X)},
\end{equation}
\begin{equation}\label{robot1}
\norm{u}_{W^{1}_2(X)} \leq C_1 e^{4R^2 \tau},
\end{equation}
whenever $\abs{x_0}<R$ where $R>0$ is large enough that $\overline X
\subset B(0,R).$
\end{proposition}

{\bf Proof.} Above we proved that the infinite series (\ref{gavnuk}) for
all sufficiently large $\tau$ is the solution to the equation
$L_q(x,D) u = 0.$  We set $r(x,x_0)=\sum_{j=2}^\infty (-1)^jU_j.$
Thanks to (\ref{lodka}) we have (\ref{robot0}).
The estimate of the first term in (\ref{lom}) follows from (\ref{kazi}).
By (\ref{kazi}) and (\ref{kazi1}), we have
\begin{eqnarray}
\sup_{x_0\in X}\Vert r(\cdot,x_0)\Vert_{L^4(X)}\le \sup_{x_0\in X}\Vert 
r(\cdot,x_0)\Vert^\frac 12 _{L^2(X)}\sup_{x_0\in X}\Vert r(\cdot,x_0)
\Vert^\frac 12 _{L^\infty(X)}\nonumber\\
\le C\frac{\Vert q\Vert_{L^p(X)}}{\tau^{\frac 12-\frac{1}{2p}}}\tau^{-1/p}
 \le  C\frac{\Vert q\Vert_{L^p(X)}}{\tau^{\frac 12+\frac{1}{2p}}}.
\end{eqnarray}
Finally estimate (\ref{robot1}) follows from (\ref{robot0}), (\ref{lom})
and the classical estimate for elliptic equations. $\blacksquare$

\section{Proof of Theorem \ref{L1}.}
We set $\tau_0=\mbox{max}\,\{\widehat C( \norm{q_1}_{L^p(X)})
\widehat C( \norm{q_2}_{L^p(X)})\}, $
where $\widehat C( \norm{q_k}_{L^p(X)})$  are determined in Proposition 
\ref{BukhSolExist2} and
let $\tau \ge\tau_0$ such that it is larger than $\tau_0$ from 
Proposition \ref{BukhSolExist2}. For point $x_0\in X$ and $\tau\ge \tau_0$ 
let $u_1 \in W^1_2(X)$ be the solution to $L_{q_1}(x,D)u_1 = 0$ given by 
Proposition \ref{BukhSolExist2}.
In particular we have
\begin{equation}\label{mk1}
u_1(x,x_0) = e^{i\tau\Phi}(1  - \frac{1}{4}e^{-i\tau (\bar\Phi+\Phi)}
\partial^{-1}(e^{i\tau (\bar\Phi+\Phi)}(\bar\partial^{-1}q_1
-\bar\partial^{-1}q_1(x_0)))+ r_1(x,x_0)),
\end{equation}
\begin{equation}
\sup_{x_0\in X} \norm{r_1(\cdot,x_0)}_{L^2(X)}\tau^{\frac 32}
+ \sup_{x_0\in X} \norm{r_1(\cdot,x_0)}_{L^4(X)}
\tau^{\frac 12+\frac{1}{2p}} \leq C \norm{q_1}_{L^p(X)},
                                                        \label{mkk1}
                                                        \end{equation}
\begin{equation}
\sup_{x_0\in X}\norm{u_1(\cdot,x_0)}_{W^{1}_2(X)} \leq
Ce^{4R^2 \tau},\label{mk51}
\end{equation}
and there exists a solution $u_2 \in W^1_2(X)$ for $L_{q_2}(x,D)u_2 = 0$
with
\begin{eqnarray}\label{mk2}
u_2(x,x_0) = e^{i\tau\overline{\Phi}}(1
- \frac{1}{4}e^{-i\tau (\bar\Phi+\Phi)}
\bar \partial^{-1}(e^{i\tau (\bar\Phi+\Phi)}(\partial^{-1}q_2
-\partial^{-1}q_2(x_0))) + r_2(x,x_0)),\\
\qquad \sup_{x_0\in X} \norm{r_2(\cdot,x_0)}_{L^2(X)}\tau^{\frac 32}
+ \sup_{x_0\in X} \norm{r_2(\cdot,x_0)}_{L^4(X)}
\tau^{\frac 12+\frac{1}{2p}}
\leq C\Vert q_2\Vert_{L^p(X)} ,\label{mkk2}\\
\sup_{x_0\in X}\norm{u_2(\cdot,x_0)}_{W^{1}_2(X)}
\leq C e^{4R^2 \tau},\label{mk50}
\end{eqnarray}
 where constant $C$ is independent of $\tau$ and $x_0$.
Substituting (\ref{mk1}) and (\ref{mk2}) into $\int_X u_1(q_1-q_2)u_2 dx$ and
using the Fubini theorem on the Cauchy-operators,  we obtain
\begin{eqnarray}\label{crab}
(q_1 - q_2)(x_0) = \left( (q_1 - q_2)(x_0)
- \int_X \frac{2 \tau}{\pi} e^{i \tau (\Phi + \overline{\Phi})}
(q_1 - q_2)(x) dx \right)\nonumber\\
+ \frac{2\tau}{\pi} \int_X u_1(q_1 - q_2)u_2 dx \nonumber\\
-\frac{2\tau}{\pi}\int_X\bar\partial^{-1}(q_1-q_2)(\partial^{-1}q_2
-\partial^{-1}q_2(x_0))e^{i\tau (\bar\Phi+\Phi)}dx\nonumber\\
-\frac{2\tau}{\pi}\int_X\partial^{-1}(q_1-q_2)(\bar\partial^{-1}q_1
-\bar\partial^{-1}q_1(x_0))e^{i\tau (\bar\Phi+\Phi)}dx\nonumber\\
- \frac{2\tau}{\pi} \int_X e^{i \tau (\Phi + \overline{\Phi})}
(q_1 - q_2)(x) (p_1p_2+r_1+r_2)(x,x_0) dx,
\end{eqnarray}
where
\begin{equation}\label{gopnik1}
p_1 = r_1 - \frac{1}{4}e^{-i\tau(\overline{\Phi}+\Phi)}
\Cab( e^{i\tau(\overline{\Phi}
+ \Phi)} ( \Ca q_1 - \Ca q_1 (x_0) ) ),
\end{equation}
\begin{equation}\label{gopnik2}
p_2 = r_2 - \frac{1}{4}e^{-i\tau(\overline{\Phi}+\Phi)}
\bar\partial^{-1}( e^{i\tau(\overline{\Phi}
+ \Phi)} ( \Cab q_2 - \Cab q_2 (x_0) ) ).
\end{equation}

We recall that $q_1 - q_2 \in \mathaccent'27{W}^s_2(X)$ by the assumptions of 
the theorem.
For $s \in (0,1] \setminus \left\{\frac{1}{2} \right\}$ and $q\in
 \mathaccent'27{W}^s_2(X)$,
let $E_0q$ be the extension in $\R^2$ by
the zero extension outside $X$.  Then
$E_0q \in W^s_2(\R^2)$.

We can now deal with the first term. Take the $L^2(X)$-norm with respect
to $x_0$ to obtain
\begin{eqnarray}
\norm{ q_1 - q_2 - \int_X \frac{2 \tau}{\pi} e^{i \tau (\Phi
+ \overline{\Phi})}(q_1 - q_2)(x) dx }_{L^2(X:,dx_0)}\nonumber \\
= \norm{ E_0(q_1 - q_2) - \int_{\R^2} \frac{2 \tau}{\pi} e^{i \tau (\Phi
+ \overline{\Phi})}E_0(q_1 - q_2)(x) dx }_{L^2(\R^2;dx_0)}.
\nonumber
\end{eqnarray}
Applying Lemma \ref{statPhaseLemma} we have
\begin{eqnarray}\label{nem0}\norm{ q_1 - q_2 - \int_X \frac{2 \tau}{\pi} 
e^{i \tau (\Phi + \overline{\Phi})}(q_1 - q_2)(x) dx }_{L^2(X;dx_0)}
\nonumber \\
\leq 2 \tau^{-s/2} \norm{E_0(q_1 - q_2)}_{W^{s}_2(\R^2)} \leq C \tau^{-s/2} 
\norm{q_1-q_2}_{W^{s}_2(X)} \leq 2CM \tau^{-s/2}.
\end{eqnarray}

\medskip
The second term on the right-hand side of (\ref{crab}) is estimated 
by the difference
of the boundary data and the definition of $d(\mathcal{C}_{q_1},
\mathcal{C}_{q_2})$:
\begin{eqnarray}\label{nem1}
\norm{\frac{2\tau}{\pi} \int_X u_1(q_1 - q_2)u_2 dx}_{L^2(X;dx_0)}
\leq C \sup_{x_0 \in X} \abs{ \frac{2\tau}{\pi} \int_X u_1(q_1 - q_2)
u_2 dx }\nonumber \\
\leq C \tau d(\mathcal{C}_{q_1}, \mathcal{C}_{q_2}) 
\sup_{x_0\in X}( \norm{u_1}_{W^{1}_2(X)} \norm{u_2}_{W^{1}_2(X)}) \leq C_{ M} 
e^{\tau(8R^2 + 1)} d(\mathcal{C}_{q_1}, \mathcal{C}_{q_2}).
\end{eqnarray}
Here in order to obtain the last estimate, we used (\ref{mk51})
and (\ref{mk50}).
Applying Lemma \ref{statPhaseLemma} again, we obtain that there exists
$\tilde s>0$ such that

\begin{equation}\label{nem5}
\left\Vert \frac{2\tau}{\pi}\int_X\bar\partial^{-1}(q_1-q_2)
(\partial^{-1}q_2-\partial^{-1}q_2(x_0))e^{i\tau (\bar\Phi+\Phi)}dx
\right\Vert_{L^2(X;dx_0)}
\end{equation}
\begin{eqnarray*}
&&\le\left\Vert \bar\partial^{-1}(q_1-q_2)\partial^{-1}q_2-\frac{2\tau}{\pi}
\int_X\bar\partial^{-1}(q_1-q_2)\partial^{-1}q_2
e^{i\tau (\bar\Phi+\Phi)}dx\right\Vert _{L^2(X;dx_0)} \\
&&+\left\Vert \bar\partial^{-1}(q_1-q_2)\partial^{-1}q_2-\partial^{-1}
q_2\frac{2\tau}{\pi}\int_X\bar\partial^{-1}(q_1-q_2)
e^{i\tau (\bar\Phi+\Phi)}dx\right\Vert_{L^2(X;dx_0)}\\
&&\le \left\Vert
 E_0\bar\partial^{-1}(q_1-q_2) E_0\partial^{-1}q_2-\frac{2\tau}{\pi}
\int_{\R^2} E_0\bar\partial^{-1}(q_1-q_2) E_0\partial^{-1}q_2
e^{i\tau (\bar\Phi+\Phi)}dx\right\Vert _{L^2(\R^2;dx_0)} \\
&&+\left\Vert
E_0\bar\partial^{-1}(q_1-q_2)E_0\partial^{-1}q_2
-E_0\partial^{-1}q_2\frac{2\tau}{\pi}
\int_{\R^2} E_0\bar\partial^{-1}(q_1-q_2)e^{i\tau (\bar\Phi+\Phi)}dx
\right\Vert_{L^2(\R^2;dx_0)}\\
&&\le \frac{C}{\tau^{\tilde s}}\Vert E_0\bar\partial^{-1} (q_1-q_2)E_0
\partial^{-1}q_2\Vert_{W^1_{2}(\mathbb{R}^2)}+\frac{C}{\tau^{\tilde s}}
\Vert E_0 \bar\partial^{-1} (q_1-q_2)\Vert_{W^1_{2}(\mathbb{R}^2)}
\Vert\partial^{-1}q_2\Vert_{L^\infty(X)}\\
&& \le \frac{C'}{\tau^{\tilde s}}\Vert q_1-q_2\Vert
_{L^2(X)}.
\end{eqnarray*}
In a similar way we obtain
\begin{eqnarray}\label{grappa}
\left\Vert \frac{2\tau}{\pi}\int_X\partial^{-1}(q_1-q_2)(\bar\partial^{-1}
q_1-\bar\partial^{-1}q_1(x_0))e^{i\tau (\bar\Phi+\Phi)}dx
\right\Vert_{L^2(X;dx_0)}\nonumber\\
\le \frac{C'}{\tau^{\tilde s}}\Vert q_1-q_2\Vert
_{L^2(X)}.
\end{eqnarray}
\medskip
 Estimating the $L^2$-norm of the last term on the righ-hand side of
(\ref{crab}), we have
\begin{eqnarray}
\mathcal I=\norm{\frac{2\tau}{\pi} \int_X
e^{i \tau (\Phi + \overline{\Phi})}(q_1 - q_2)(x)(p_1p_2 + r_1 + r_2 )(x,x_0)
dx}_{L^2(X;dx_0)}
\nonumber \\
\leq C \sup_{x_0\in X} \frac{2\tau}{\pi} \int_X
\vert (q_1 - q_2)(x)\vert \vert(p_1p_2+r_1+r_2 )(x,x_0)\vert dx.
\nonumber
\end{eqnarray}
Thanks to (\ref{mkk1}) and (\ref{mkk2}), we obtain
\[
\mathcal I \le C\tau \Vert {q_1 - q_2}\Vert_{L^2(X)}
\sup_{x_0\in X}\Vert (p_1 p_2 +r_1 + r_2)(\cdot,x_0)\Vert_{L^2(X)}
\]
$$
\le C\tau \Vert {q_1 - q_2}\Vert_{L^2(X)}
\sup_{x_0\in X}(\Vert p_1 p_2 \Vert_{L^2(X)}+\Vert (r_1 + r_2)(\cdot,x_0)\Vert
_{L^2(X)})
$$
$$
\le C_1\Vert {q_1 - q_2}\Vert_{L^2(X)}\sup_{x_0\in X}(\tau\Vert p_1 p_2 
\Vert_{L^2(X)}+\frac {1}{\root\of\tau}).
$$

By (\ref{mkk1}), (\ref{mkk2}) and Proposition \ref{BukhSolExist2}
\begin{eqnarray}
\sup_{x_0\in X}\Vert p_1 p_2 \Vert_{L^2(X)}
&\le& \sup_{x_0\in X}(\Vert r_1\Vert_{L^4(X)}
\Vert r_2 \Vert_{L^4(X)}\nonumber                \\
&&+ \frac{1}{4}\Vert \Cab( e^{i\tau(\overline{\Phi}
+ \Phi)} ( \Cab q_2 - \Cab q_2 (x_0) ) )\Vert_{L^\infty(X)}\Vert r_1\Vert
_{L^2(X)}\nonumber\\
&&+ \frac{1}{4}
\Vert  \Ca( e^{i\tau(\overline{\Phi}
+ \Phi)} ( \Ca q_1 - \Ca q_1 (x_0) ) )\Vert_{L^\infty(X)}
\Vert r_2 \Vert_{L^2(X)}\nonumber\\
&&+ \frac{1}{16}\Vert  \Ca( e^{i\tau(\overline{\Phi}
+ \Phi)} ( \Ca q_1 - \Ca q_1 (x_0) ) )\Vert_{L^2(X)} \nonumber\\
&&\Vert \Cab( e^{i\tau(\overline{\Phi}
+ \Phi)} ( \Cab q_2 - \Cab q_2 (x_0) ) )\Vert_{L^\infty(X)})\nonumber\\
&\le& C\Big(
\frac {1}{\tau^\frac 32}+\frac{1}{\tau^p}(\Vert r_1\Vert_{L^2(X)}
+\Vert r_2\Vert_{L^2(X)}) \nonumber\\
&&+ \frac{1}{\tau^p} \Vert  \Ca( e^{i\tau(\overline{\Phi}
+ \Phi)} ( \Ca q_1 - \Ca q_1 (x_0) ) )\Vert_{L^2(X)}\Big).
                                          \nonumber
\end{eqnarray}
Applying (\ref{mkk1}), (\ref{mkk2}) and Proposition \ref{elka} with 
$\epsilon=\frac p2$, we obtain:
\begin{equation}
\sup_{x_0\in X}\Vert p_1 p_2 \Vert_{L^2(X)}\le C(\frac {1}{\tau^\frac 32}
+\frac{1}{\tau^p}(\frac{1}{\tau^\frac 32}+\frac{1}{\tau^{1-\frac p2}})).
\end{equation}

Hence there exists $\tau_1$ independent of $z_0$ such that
\begin{equation}\label{32}
\mathcal I\le \frac{1}{2}\Vert {q_1 - q_2}\Vert_{L^2(X)}
\quad\forall\tau\ge\tau_1.
\end{equation}

\medskip
Combining estimates (\ref{nem0})-(\ref{32}) and setting
$R_0 = 8R^2+1$, we obtain
\begin{equation}\label{nem3}
\Vert q_1 - q_2 \Vert_{L^2(X)}
\le C(e^{\tau R_0}d(\mathcal{C}_{q_1}, \mathcal{C}_{q_2})
+ \tau^{-s/2}), \quad \forall \tau \ge \tau_1.
\end{equation}
Replacing $\tau$ and $C$ by $\tau+\tau_1$ and $Ce^{R_0\tau_1}$ respectively,
we have (\ref{nem3}) for all $\tau>0$.  For obtaining the conditional
stability, we should make the right-hand side of (\ref{nem3})
as small as possible by choosing $\tau>0$.
For this we make the following choice of $\tau$ depending on the value of 
$d(\mathcal{C}_{q_1}, \mathcal{C}_{q_2}).$
\\
{\bf Case 1:} $d(\mathcal{C}_{q_1}, \mathcal{C}_{q_2}) < 1$.\\
We choose
$$
\tau = \frac{\alpha}{R_0}\left( 1+
\ln \frac{1}{d(\mathcal{C}_{q_1}, \mathcal{C}_{q_2})}\right) > 0
$$
with arbitrarily fixed $\alpha \in (0,1)$.  Then
$e^{\tau R_0}d(\mathcal{C}_{q_1}, \mathcal{C}_{q_2})
= e^{\alpha}d(\mathcal{C}_{q_1}, \mathcal{C}_{q_2})^{1-\alpha}$ and
$$
\tau^{-s/2} = \left(\frac{R_0}{\alpha}\right)^{s/2}
\left( 1+ \ln \frac{1}{d(\mathcal{C}_{q_1}, \mathcal{C}_{q_2})}
\right)^{-s/2}.
$$
Since for $0 < \alpha < 1$, there exists a constant $C>0$ such that
$\eta^{1-\alpha} \le C\left( 1+  \ln \frac{1}{\eta}\right)^{-s/2}$ for
$0 \le \eta < 1$, with this choice of $\tau$, estimate (\ref{nem3}) yields
$$
\Vert q_1-q_2\Vert_{L^2(X)}
\le C\left( 1+\ln \frac{1}{d(\mathcal{C}_{q_1}, \mathcal{C}_{q_2})}
\right)^{-s/2}.
$$
\\
{\bf Case 2:} $d(\mathcal{C}_{q_1}, \mathcal{C}_{q_2}) \ge 1$.\\
Since  $\Vert q_1\Vert_{W^s_2(X)}\le M$ and $\Vert q_2\Vert_{W^s_2(X)} \le M$,
we have
$\Vert q_1 - q_2\Vert_{L^2(X)} \le 2M \le 2M
d(\mathcal{C}_{q_1}, \mathcal{C}_{q_2})$.

Therefore combining the two cases, we complete the proof of Theorem 2.1.
$\blacksquare$

\section{\bf Proof of theorem \ref{L2}.}
For any point $x_0\in X$ let $u_1,u_2 \in W^1_2(X)$ be the solutions
to the Schr\"odinger equation given by (\ref{mk1}) and (\ref{mk2})
respectively.

Since the Dirichlet-to-Neumann maps are the same,
we have $\int_X (q_1-q_2)u_1u_2dx=0$.
Then plugging formulas (\ref{mk1}) and (\ref{mk2}) into it 
and adding $(q_1-q_2)(x_0)$ to both sides, we have

\begin{eqnarray}\label{zoloto}
(q_1 - q_2)(x_0) = \left( (q_1 - q_2)(x_0)
- \int_X \frac{2 \tau}{\pi} e^{i \tau (\Phi + \overline{\Phi})}(q_1 - q_2)(x) dx \right)\nonumber\\
-\frac{2\tau}{\pi}\int_X\bar\partial^{-1}(q_1-q_2)(\partial^{-1}q_2-\partial^{-1}q_2(x_0))e^{i\tau (\bar\Phi+\Phi)}dx\nonumber\\
-\frac{2\tau}{\pi}\int_X\partial^{-1}(q_1-q_2)(\bar\partial^{-1}q_1-\bar\partial^{-1}q_1(x_0))e^{i\tau (\bar\Phi+\Phi)}dx\nonumber\\
- \frac{2\tau}{\pi} \int_X e^{i \tau (\Phi + \overline{\Phi})}
(q_1 - q_2)(x) (p_1p_2+r_1+r_2)(x,x_0) dx,
\end{eqnarray}
where the functions $p_j$ are determined by (\ref{gopnik1}) and (\ref{gopnik2}).

Since the estimates (\ref{nem5}), (\ref{32}) hold true for all sufficiently
large $\tau$, 
we obtain from (\ref{zoloto}):
\begin{eqnarray}
\Vert q_1-q_2\Vert_{L^2(X)}\le C \norm{ q_1 - q_2 - \int_X \frac{2 \tau}{\pi} e^{i \tau (\Phi + \overline{\Phi})}(q_1 - q_2)(x) dx }_{L^2(X;dx_0)}
\nonumber \\
= C \norm{ E_0(q_1 - q_2) - \int_{\R^2} \frac{2 \tau}{\pi} e^{i \tau (\Phi + \overline{\Phi})}E_0(q_1 - q_2)(x) dx }_{L^2(\R^2;dx_0)}.\nonumber
\end{eqnarray}
In view of
Lemma \ref{statPhaseLemma} we obtain
\begin{equation}\label{1nem0}\nonumber\norm{ q_1 - q_2 - \int_X \frac{2 \tau}{\pi} e^{i \tau (\Phi + \overline{\Phi})}(q_1 - q_2)(x) dx }_{L^2(X;dx_0)}
\rightarrow 0\,\,\mbox{as}\,\,\tau\rightarrow +\infty.
\end{equation}
The proof of the theorem is complete.
$\blacksquare$

{\bf Acknowledgement}.  The authors thank the anonymous referees
for valuable comments.

\end{document}